\documentclass[a4paper,11pt,english]{smfart}
\usepackage[applemac]{inputenc}
\usepackage[francais,english]{babel}
\usepackage{amsfonts}
\usepackage{amsmath} 
\usepackage{hyperref} 
\usepackage{latexsym}
\usepackage{array}
\usepackage{amssymb}
\usepackage{enumerate}
\usepackage{smfthm}
\usepackage{graphicx}
\theoremstyle{plain}

\newtheorem{definition}{Definition}

\renewcommand{\l}{  \ell  }

\newcommand{\R}{  \mathbb{R}   }

\newcommand{\eps}{\varepsilon}
\newcommand{\e}{  \mathrm{e}   }

\newcommand{\Z}{  \mathbb{Z}   }
\newcommand{\N}{  \mathbb{N}   }

\newcommand{\A}{  \mathcal{A}   }

\newcommand{\T}{  \mathbb{S}^1 }

\newcommand{\dis}{\displaystyle}

\newcommand{\ov}{  \overline  }
\renewcommand{\a}{  \alpha   }
\renewcommand{\b}{  \beta   }

\renewcommand{\phi}{  \varphi  }
\renewcommand{\L}{  \mathcal{L}   }
\newcommand{\wh}{  \widehat   }

\renewcommand{\S}{  \mathcal{S}  }
\numberwithin{equation}{section}
 
 \author{ Emanuele Haus}
\address{Laboratoire de Math\'ematiques J. Leray, Universit\'e de Nantes, UMR CNRS 6629\\
2, rue de la Houssini\`ere \\
44322 Nantes Cedex 03, France.}
\email{emanuele.haus@univ-nantes.fr}

\author{ Laurent Thomann }
\address{Laboratoire de Math\'ematiques J. Leray, Universit\'e de Nantes, UMR CNRS 6629\\
2, rue de la Houssini\`ere \\
44322 Nantes Cedex 03, France.}
\email{laurent.thomann@univ-nantes.fr}

\title[Dynamics on resonant clusters for NLS]
{Dynamics on resonant clusters for the quintic non linear Schr\"odinger equation}

\begin{document}

\begin{abstract}
We construct solutions to the quintic nonlinear Schr\"odinger equation   on the circle
$$  i\partial_t u+\partial_{x}^{2}u = \nu \ |u|^4u,\quad \nu\ll1, \ x\in \mathbb{S}^{1},\ t\in \R,$$
with initial conditions supported on arbitrarily many different resonant clusters. This is a sequel  the work \cite{GT2} of Beno\^it Gr\'ebert and the second author.
    \end{abstract}

\altkeywords{Forme Normale, Equation de Schr\"odinger non linÈaire,
rÈsonances, Èchange d'Ènergie}\frontmatter
\thanks{
\noindent Both authors were supported in part by the  grant ANR-10-JCJC 0109.}

\maketitle

\section{Introduction and results}
 \subsection{General introduction}
In this paper we consider   the quintic nonlinear periodic Schr\"odinger equation
\begin{equation}\label{cauchy} 
\left\{
\begin{aligned}
&i\partial_t u+\partial_{x}^{2}u  =  \nu|u|^{4}u,\quad
(t,x)\in\R\times \mathbb{S}^1,\\
&u(0,x)= u_{0}(x),
\end{aligned}
\right.
\end{equation} 
 where $\nu>0$ is some small parameter. In \cite{GT2} B. Gr\'ebert  and the second author showed  a beating effect for \eqref{cauchy}: there exist solutions which are supported on 4 so-called resonant modes and which stay close to a time-periodic solution for long time. These solutions moreover  show that there is an energy exchange between the considered modes, which is a genuine nonlinear effect. We call this mechanism a beating effect. Such a phenomenon was first observed by   Gr\'ebert and Villegas-Blas  \cite{GV} on a cubic Schr\"odinger equation. Let us also mention the work \cite{GPT} by Gr\'ebert-Paturel-Thomann where a general principle is extracted so that such a phenomenon occurs. Finally we refer to the introduction of \cite{GT2} for more results on the long time dynamics for \eqref{cauchy} and related models.

Our aim here is to extend the main result of \cite{GT2} and show that we can construct solutions to \eqref{cauchy} supported on arbitrarily many resonant modes.

Recall \cite[Definition 1.1]{GT2} that a set of the form 
\begin{equation}\label{res}
\big\{n,n+3k,n+4k,n+k\big\}, \quad k\in \Z\backslash\{0\}\; \text{ and } n\in\Z,
\end{equation}
is called  a resonant set. These sets exactly correspond to the  resonant monomials of order 6 of the Hamiltonian of \eqref{cauchy} which contain 4 different modes. Here, we consider resonant sets of the type
\begin{equation}\label{Ak} 
\mathcal{A}_{k}=\big\{n_{k}-2,n_{k}-1,n_{k}+1,n_{k}+2\big\},
\end{equation}
where  $(n_{k})_{k\geq 1}$ is  a sequence of integers which will be described.  In the sequel, we denote by
\begin{equation}\label{1.3}
 a^{(k)}_{2}=n_{k}-2, \quad a^{(k)}_{1}=n_{k}+1,\quad  b^{(k)}_{2}=n_{k}+2, \quad b^{(k)}_{1}=n_{k}-1,
\end{equation}
and for $K\geq 1$, we define
\begin{equation*}
\mathcal{A}=\bigcup_{k=1}^{K}\mathcal{A}_{k}.
\end{equation*}

To begin with, let us recall the result of \cite{GT2} for the resonant set $\A_{k}$.

  \begin{theo}[\cite{GT2}]\label{thm0}
Let $k\geq 1$.  There exist $T_{k}>0$, $\nu_{0}>0$, $\a_{k}\in(0,1/2)$ and a $2T_k-$periodic function $K^{(k)}_{\star}:\R\longmapsto (0,1)$ which satisfies $K^{(k)}_{\star}(0)\leq \a_{k}$ and $K^{(k)}_{\star}(T_{k})\geq1-\a_{k}$ and if $0<\nu<\nu_{0}$, there exists a solution $u_{k}$ to \eqref{cauchy} satisfying for all $|t|\leq \nu^{-9/8}$
\begin{equation*}
u_{k}(t,x)=v_{k}(t,x)+\nu^{1/4} q(t,x),
\end{equation*}
with 
\begin{equation*}
 v_{k}(t,x)=\sum_{j\in \mathcal{A}_{k}}w_{j}(t)\e^{ijx},
  \end{equation*}
\vspace{-0,4 cm}
and
\begin{equation*} 
\begin{array}{ccccc}
 |w_{a^{(k)}_{1}}(t)|^{2}&=&2 |w_{a^{(k)}_{2}}(t)|^{2}&=&K^{(k)}_{\star}(\nu t)\\[5pt]
 |w_{b^{(k)}_{1}}(t)|^{2}&=&2 |w_{b^{(k)}_{2}}(t)|^{2}&=&1-K^{(k)}_{\star}(\nu t),
 \end{array}
 \end{equation*}
 and $q$ is smooth in time and analytic in space on $[-\nu^{-9/8},\nu^{-9/8}]\times \mathbb{S}^{1} $. Moreover, the Fourier coefficients $\wh{q}_{j}(t)$ of $q(t)$ satisfy 
\begin{equation*}
\sup_{|t|\leq \nu^{-9/8}}|\wh{q}_{j}(t)|\leq C \e^{-|j|},
\end{equation*}
with $C$ independent of $k\geq 1$ and $\nu>0$.
 \end{theo}

The result is not exactly stated like this in \cite{GT2}, but is proven there. In particular, the analycity of the remainder term follows from the analytical framework of the Birkhoff normal form procedure in \cite[Section 3]{GT2}. See also \cite{GPT}.

Theorem \ref{thm0} shows that there are non trivial interactions between the modes in $\A_{k}$ and they occur for times $t\sim \nu^{-1} T_{k}$.

\subsection{The main result}
 
In this paper we prove that if the resonant sets are carefully chosen, there exist  solutions to \eqref{cauchy} which are the superposition of solutions of the previous type.

\begin{theo}\label{thm1}
There exists $\nu_0>0$ and there exists an increasing sequence of integers $(n_{k})_{k\geq 1}$ such that, if $0<\nu<\nu_0$, for all $K\geq 1$ with $n_{K} \leq -c\ln \nu$ there exists a solution to \eqref{cauchy} which reads for all $|t|\leq \nu^{-9/8}$
\begin{equation}\label{u}
 u(t,x)=\sum_{k=1}^{K}\e^{-n_{k}}v_{k}(t,x)+\nu^{1/4} q(t,x),
  \end{equation}
where
\begin{enumerate}[i)]
\item For all $1\leq k\leq K$, $v_{k}$ is as in Theorem \ref{thm0}.
\item The error term is smooth in time and analytic in space on $[-\nu^{-9/8},\nu^{-9/8}]\times \mathbb{S}^{1} $. Moreover, the Fourier coefficients $\wh{q}_{j}(t)$ of $q(t)$ satisfy 
\begin{equation*}
\sup_{|t|\leq \nu^{-9/8}}|\wh{q}_{j}(t)|\leq C \e^{-|j|},
\end{equation*}
with $C$ independent of $K\geq 1$ and $\nu>0$.
\end{enumerate}
\end{theo}

This result shows  a beating effect inside of each  resonant set, but there is no energy transfer between two different resonant clusters. In particular, we do not show an energy transfer from the low to the high frequencies.

In our example, for all $s\geq 0$, $\|u\|_{H^{s}}$ is almost preserved during the time. This is due to the particular form \eqref{Ak} of our resonant sets. We believe that a similar construction  for more general resonant sets \eqref{res} can be done.

For all $j\in \Z$, the Fourier coefficient $\wh{u}_{j}$ of $u$ in \eqref{u} satisfies 
\begin{equation*}
\sup_{|t|\leq \nu^{-9/8}}|\wh{u}_{j}(t)|\leq C \e^{-|j|},
\end{equation*}
thus $u$ is bounded in an analytic norm uniformly in $K\geq 1$ for this time scale. A natural question is whether we can choose $K=+\infty$ in Theorem \ref{thm1}. Our method does not allow this extension since the period of $v_{k}$ grows to infinity with $k$. Moreover, the expansion in \eqref{u} is relevant as long as $\e^{-n_{k}}v_{k}$ is larger than the error term, and this gives the limitation $n_{K}\leq -c\ln \nu$.

In fact there are many sequences which satisfy Theorem \ref{thm1}: almost all sequences which satisfy $n_{k+1}\geq 12 n^{2}_{k}$ can be taken (see the proof of Proposition \ref{prop.nonres}).

Our approach allows also to treat the  focusing Schr\"odinger equation ($\nu<0$), but for simplicity we only deal with the case $\nu>0$.\\

  With an appropriate choice of the initial conditions, we can construct  quasi-periodic solutions in time for a large set of frequencies. 
  
  \begin{coro}\label{Coro}
For every $K\geq 1$ and every sequence of real numbers $\Lambda_{1},\Lambda_{2},\dots,\Lambda_{K}>0$, there exists $N\in\N$ so that, if $\nu>0$ is small enough, we can construct $v_{k}$ of period  $2N\Lambda_{k}/\nu$  in time.
\end{coro}

This is clearly a nonlinear phenomenon, since in the linear regime all the frequencies are integer multiples of the same number.

 \subsection{Plan of the paper}

In Section \ref{Sect2} we prove the existence of resonant sets made up of several clusters which do not interact much with one another. In Section \ref{Sect3} we recall the Hamiltonian structure of \eqref{cauchy} and we study the model equation, which is obtained by truncating the error terms of the normal form (higher order terms and terms involving frequencies outside the resonant sets). In Section \ref{Sect4} we perform the perturbation analysis and collect the results of the previous sections in order to prove our main results. Throughout the paper, we widely rely on the results obtained by B. Gr\'ebert and the second author in \cite{GT2}: since several proofs turn out to be similar, here we choose to highlight what is new and different, rather than copy out the proofs in \cite{GT2}.

\section{Existence of resonant sets}\label{Sect2}
In this section, we show the existence of a sequence $(n_{k})$ such that the modes in $\A_{k}$ and $\A_{j}$ do not interact much when $k<j$. We will see that this is ensured when $(n_{k})$ is growing fast enough and if it satisfies  some arithmetical condition.

 \begin{lemm}\label{lem.couple}
Assume that the sequence $(n_{k})_{k\geq 1}$ satisfies $n_{1}\geq 3$ and $n_{k+1}\geq 12n^{2}_{k}$. Then the following holds true: Let  $\mathcal{S}:= \{j_1,j_2,j_3,\ell_1,\ell_{2},\ell_{3}\} \subset \mathcal{A}   $ with 
 \begin{equation} \label{eq.res}
\left\{
\begin{aligned}
& j_{1}^{2}+j_{2}^{2}+j^{2}_{3}  = \ell^{2}_1+\ell^{2}_{2}+\ell^{2}_{3} ,\\
& j_{1}+j_{2}+j_{3}=\ell_1+\ell_{2}+\ell_{3},
\end{aligned} 
\right. \quad \text{and}\quad \big\{j_{1},j_{2},j_{3}\big\}\neq \big\{\ell_{1},\ell_{2},\ell_{3}\big\}.
\end{equation}
Then there exists $1\leq k\leq K$ so that $\S\subset \A_{k}$.
 \end{lemm}
 
  \begin{proof}
  Consider $\S=\{j_1,j_2,j_3,\ell_1,\ell_{2},\ell_{3}\} \subset \A$ such that \eqref{eq.res} holds. Assume that $\l_{1}=\max \S$. Then $\l_{1}=n_{k}+r_{1}$ for some $1\leq k\leq K$ and $r_{1}\in \{-2,-1,1,2\}$. If $k=1$ we are done, hence we assume that $k\geq 2$. We claim that one of the integers $j_{1},j_{2},j_{3}$, say $j_{1}$, is of the form $j_{1}=n_{k}+q_{1}$ with  $q_{1}\in \{-2,-1,1,2\}$. If it is not the case, $j_{1},j_{2},j_{3}\leq n_{k-1}+2$ and thus 
  \begin{equation*}
  j^{2}_{1}+  j^{2}_{2}+  j^{2}_{2}\leq 3(n_{k-1}+2)^{2}< 12 n^{2}_{k-1}\leq n_{k}\leq (n_{k}-2)^{2}\leq \l^{2}_{1},
  \end{equation*}
 which is a contradiction. We plug the expressions of $\l_{1}$ and $j_{1}$ in \eqref{eq.res} and obtain
 \begin{equation}\label{25}
 2(q_{1}-r_{1})n_{k}=\l^{2}_{2}+\l^{2}_{3}-j^{2}_{2}-j^{2}_{3}+r^{2}_{1}-q^{2}_{1}
 \end{equation}
From \cite[Lemma 2.1]{GT2} we have  $q_{1}\neq r_{1}$. \\
$\bullet$ Assume that $\{\l_{2},\l_{3},j_{2},j_{3}\}\subset \bigcup_{m=1}^{k-1}\A_{m}$. Then 
\begin{eqnarray*}
|\l^{2}_{2}+\l^{2}_{3}-j^{2}_{2}-j^{2}_{3}+r^{2}_{1}-q^{2}_{1}|&\leq &6(n_{k-1}+2)^{2}\\
&<&24n^{2}_{k-1}\\
&\leq & 2n_{k}\\
&\leq &2|q_{1}-r_{1}|n_{k},
\end{eqnarray*}
  which is a contradiction.\\
  $\bullet$ We can therefore assume that $\l_{2}\in \A_{k}$. If $\{j_{2},j_{3}\}\subset \bigcup_{m=1}^{k-1}\A_{m}$ we can show that 
  \begin{equation*}
 2(q_{1}-r_{1})n_{k}<\l^{2}_{2}+\l^{2}_{3}-j^{2}_{2}-j^{2}_{3}+r^{2}_{1}-q^{2}_{1},
  \end{equation*}
  thus we can assume that $j_{2}\in \A_{k}$ and we write $j_{2}=n_{k}+q_{2}$, $\l_{2}=n_{k}+r_{2}$. The relation \eqref{25} then reads 
   \begin{equation*} 
 2(q_{1}+q_{2}-r_{1}-r_{2})n_{k}=\l^{2}_{3}-j^{2}_{3}+r^{2}_{1}+r^{2}_{2}-q^{2}_{1}-q^{2}_{2}.
 \end{equation*}
 With a similar argument we deduce that $\{\l_{3},j_{3}\}\subset \A_{k}$, which completes the proof.
  \end{proof}

 Define the set 
   \begin{equation*}
  {\mathcal R}=\big\{\,(j_1,j_2,j_3,\ell_1,\ell_{2},\ell_{3})\in\mathbb{Z}^6\;\;s.t. \quad 
 j_1+j_2+j_3=\ell_{1}+\ell_{2}+\ell_{3} \;\; {\rm and } \;\; j_1^2+j_2^2+j_3^2=\ell^{2}_{1}+\ell^{2}_{2}+\ell^{2}_{3}\,\big\}.
  \end{equation*}
 \begin{prop}\label{prop.nonres}
 There exists a sequence $(n_{k})_{k\geq 1}$ which satisfies $n_{1}\geq 3$,  $n_{k+1}\geq 12 n^{2}_{k}$ and so that the following holds true: Let  $(j_{1},j_{2},j_{3},\ell_{1},p_{1},p_{2})\in \mathcal{R}$. Assume that $j_{1},j_{2},j_{3},\ell_{1}\in \mathcal{A}$. Then $p_{1},p_{2}\in \mathcal{A}$.
 \end{prop}

 \begin{proof}
We construct such a sequence $(n_{k})_{k\geq 1}$ by induction. When $K=1$, we can apply  \cite[Lemma 2.4]{GT2} and set any $n_{1}\geq 3$.
Now, assume that we have constructed the first $K$ elements of the sequence $(n_{k})_{k=1}^K$. We will prove that we can choose $n_{K+1}$ satisfying the wanted properties.

Suppose that we have fixed $n_{K+1}$ (and therefore $\mathcal{A}_{K+1}$): we now investigate which arithmetical properties are required in order to satisfy the non-resonance condition.

Let $j_{1},j_{2},j_{3},\ell_{1}\in \mathcal{A}$ and $p_{1},p_{2}\in \mathbb{N}$ so that 
  \begin{equation}\label{eq.double}
\left\{
\begin{aligned}
& p_{1}+p_{2}=j_{1}+j_{2}+j_{3}-\ell_{1}=:S   , \\   
& p^{2}_{1}+p^{2}_{2}=j^{2}_{1}+j^{2}_{2}+j^{2}_{3}-\ell^{2}_{1}=:T  . 
\end{aligned}
\right.
\end{equation}
The two complex (possibly coinciding) solutions $(p_1,p_2)$ to \eqref{eq.double} are the roots of the polynomial
\begin{equation*}
X^2-SX+\frac{1}{2}(S^2-T)\ .
\end{equation*}
The discriminant of this polynomial is $\Delta=2T-S^2$. Therefore, a necessary condition for \eqref{eq.double} to have integer solutions is that $2T-S^2$ is a perfect square.

Each of the elements $j_{1},j_{2},j_{3},\ell_{1}$ may belong either to $\mathcal{A}_{K+1}$ or to $\mathcal{A}_{k}$ with $k\leq K$. We have to distinguish 8 different cases, depending on how many of the $j$'s belong to $\mathcal{A}_{K+1}$ (4 possibilities, from 0 to 3) and whether $\ell_1\in\mathcal{A}_{K+1}$ or not.

$\bullet$ Case (0,0): $j_{1},j_{2},j_{3},\ell_{1}\notin\mathcal{A}_{K+1}$.\\
No further property has to be verified: the non-resonance condition is satisfied by induction hypothesis.

$\bullet$ Case (1,0): $j_{1}\in\mathcal{A}_{K+1}$, $j_{2},j_{3},\ell_{1}\notin\mathcal{A}_{K+1}$.\\
We exploit the identity
\begin{equation*}
\Delta=2(j^{2}_{1}+j^{2}_{2}+j^{2}_{3}-\ell^{2}_{1})-(j_{1}+j_{2}+j_{3}-\ell_{1})^2=(j_{1}-j_{2}-j_{3}+\ell_{1})^2-4(\ell_1-j_2)(\ell_1-j_3) .
\end{equation*}
Now, $j_1$ may be expressed as $j_1=n_{K+1}+c_1$, with $c_1\in\{-2,-1,1,2\}$. Therefore
\begin{equation*}
\Delta=(n_{K+1}+c_1-j_{2}-j_{3}+\ell_{1})^2-4(\ell_1-j_2)(\ell_1-j_3).
\end{equation*}
The relevant thing here is that $\Delta$ has the form $(n_{K+1}+\tilde{c}_1)^2+\tilde{c}_2$. If $\tilde{c}_2=0$, then either $\ell_1=j_2$ or $\ell_1=j_3$, which implies the non-resonance condition $p_1,p_2\in\mathcal{A}$, by \cite[Lemma 2.1]{GT2}. If $\tilde{c}_2\neq 0$, it is sufficient to choose $n_{K+1}$ large enough to prevent $\Delta$ from being a perfect square.

$\bullet$ Case (2,0): $j_{1},j_{2}\in\mathcal{A}_{K+1}$, $j_{3},\ell_{1}\notin\mathcal{A}_{K+1}$.\\
We have $j_1=n_{K+1}+c_1$, $j_2=n_{K+1}+c_2$ with $c_1,c_2\in\{-2,-1,1,2\}$. Thus,
\begin{eqnarray*}
\nonumber \Delta&=&2\left[(n_{K+1}+c_1)^2+(n_{K+1}+c_2)^2+j_3^2-\ell_1^2\right]-\left[(n_{K+1}+c_1)+(n_{K+1}+c_2)+j_3-\ell_1\right]^2=\\
&=&-4(j_3-\ell_1)n_{K+1}+(c_1-c_2)^2-2(c_1+c_2)(j_3-\ell_1)+(j_3-\ell_1)(j_3+3\ell_1) .
\end{eqnarray*}
If $\ell_1=j_3$, then $p_1,p_2\in\mathcal{A}$, by \cite[Lemma 2.1]{GT2}. If $\ell_1\neq j_3$, then $\Delta$ has the form $\a n_{K+1}+\b$, with $\a\neq 0$.

$\bullet$ Case (3,0): $j_{1},j_{2},j_{3}\in\mathcal{A}_{K+1}$, $\ell_{1}\notin\mathcal{A}_{K+1}$.\\
The conditions $n_{1}\geq 3$,  $n_{k+1}\geq 12 n^{2}_{k}$ imply that if $k_1\neq k_2$ then $|n_{k_1}-n_{k_2}|\geq 105=12\cdot 3^2-3$. Therefore we exploit the translation invariance of the resonance condition and we translate back $(j_{1},j_{2},j_{3},\ell_{1},p_{1},p_{2})$ obtaining the new resonant sextuple
\[(\tilde j_{1},\tilde j_{2},\tilde j_{3},\tilde\ell_{1},\tilde p_{1},\tilde p_{2}):=(j_{1}-n_{K+1},j_{2}-n_{K+1},j_{3}-n_{K+1},\ell_{1}-n_{K+1},p_{1}-n_{K+1},p_{2}-n_{K+1})\]
with $|\tilde j_1|,|\tilde j_2|,|\tilde j_3|\leq 2$ and $|\tilde\ell_1|\geq 105-2=103$, which is clearly absurd since
\[\tilde j_1^2+\tilde j_2^2+\tilde j_3^2=\tilde\ell^{2}_{1}+\tilde p^{2}_{1}+\tilde p^{2}_{2} .\]

$\bullet$ Case (0,1): $\ell_{1}\in\mathcal{A}_{K+1}$, $j_{1},j_{2},j_{3}\notin\mathcal{A}_{K+1}$.\\
This case is easily seen to be absurd, since
\[j_1^2+j_2^2+j_3^2=\ell^{2}_{1}+p^{2}_{1}+p^{2}_{2}\]
and $\ell_1$ is much bigger than $j_1,j_2,j_3$.

$\bullet$ Case (1,1): $j_{1},\ell_{1}\in\mathcal{A}_{K+1}$, $j_{2},j_{3}\notin\mathcal{A}_{K+1}$.\\
We write $j_1=n_{K+1}+c_1$, $\ell_1=n_{K+1}+c_2$ with $c_1,c_2\in\{-2,-1,1,2\}$. So we have
\begin{eqnarray*}
\nonumber \Delta&=&2\left[(n_{K+1}+c_1)^2+j_2^2+j_3^2-(n_{K+1}+c_2)^2\right]-\left[(n_{K+1}+c_1)+j_2+j_3-(n_{K+1}+c_2)\right]^2=\\
&=&4(c_1-c_2)n_{K+1}+(c_1^2+j_2^2+j_3^2-c_2^2)-(c_1+j_2+j_3-c_2)^2 .
\end{eqnarray*}
If $c_1=c_2$, then $j_1=\ell_1$ and $p_1,p_2\in\mathcal{A}$ because of \cite[Lemma 2.1]{GT2}. Otherwise, $\Delta$ has the form $\a n_{K+1}+\b$, with $\a\neq 0$.

$\bullet$ Case (2,1): $j_{1},j_{2},\ell_{1}\in\mathcal{A}_{K+1}$, $j_{3}\notin\mathcal{A}_{K+1}$.\\
We have $j_1=n_{K+1}+c_1$, $j_2=n_{K+1}+c_2$, $\ell_1=n_{K+1}+c_3$ with $c_1,c_2,c_3\in\{-2,-1,1,2\}$. Therefore
\begin{eqnarray*}
\nonumber \Delta&=&2\left[(n_{K+1}+c_1)^2+(n_{K+1}+c_2)^2+j_3^2-(n_{K+1}+c_3)^2\right]-(n_{K+1}+c_1+c_2+j_3-c_3)^2=\\
&=&(n_{K+1}+c_1+c_2-c_3-j_3)^2-4(c_3-c_1)(c_3-c_2) ,
\end{eqnarray*}
which has the same structure as for the case (1,0) and therefore $\Delta$ is not a perfect square provided that $n_{K+1}$ is large enough.

$\bullet$ Case (3,1): $j_{1},j_{2},j_{3},\ell_{1}\in\mathcal{A}_{K+1}$.\\
Then we have directly $p_1,p_2\in\mathcal{A}_{K+1}\subset\mathcal{A}$, by \cite[Lemma 2.4]{GT2}.\\

Now, what still has to be proved is that we can choose $n_{K+1}$ arbitrarily large and such that $\a_rn_{K+1}+\b_r$ is not a square for a finite number of couples of integers $\{(\a_r,\b_r)\}_{r=1}^s$. We can limit ourselves to $\a_r>0$, since the conditions to be satisfied yield $\a_r\neq 0$ and, if $\a_r<0$, then $\a_rn_{K+1}+\b_r$ is negative and therefore not a perfect square, for $n_{K+1}$ large enough. Let $\a:=(\a_1,\ldots,\a_s)\in(\mathbb{N}\setminus\{0\})^s$ and $\b:=(\b_1,\ldots,\b_s)\in\mathbb{Z}^s$. We denote by $S_{\a\b}:=\{(\a_r,\b_r)\}_{r=1}^s$ the set of all couples $(\a_r,\b_r)$, for a given choice of $\a\in(\mathbb{N}\setminus\{0\})^s$ and $\b\in\mathbb{Z}^s$.

\begin{definition}
We say that a positive integer $n\in\mathbb{N}$ satisfies the ``no-square condition'' with respect to $S_{\a\b}$ (NSC-$S_{\a\b}$) if for all $r=1\ldots s$, $\a_rn+\b_r$ is not a perfect square.
\end{definition}

Fix $\a\in(\mathbb{N}\setminus\{0\})^s$, $\b\in\mathbb{Z}^s$. Let $F_N$ be the number of positive integers $1\leq n\leq N$ which satisfy (NSC-$S_{\a\b}$).

Consider a single couple $(\a_r,\b_r)$: the main result in \cite{BZ02} implies that there exist two universal constants $C_1,C_2$ (which do not depend on $\a_r,\b_r$), such that the number of positive integers $1\leq n\leq N$ such that $\a_rn+\b_r$ is a perfect square is at most $C_1N^{3/5}(\ln N)^{C_2}$. Now, a positive integer $n$ fails to satisfy (NSC-$S_{\a\b}$) if and only if $\a_rn+\b_r$ is a perfect square for at least one of the $s$ couples $(\a_r,\b_r)$. Therefore, we deduce that the number of positive integers $\leq N$ which fail to satisfy (NSC-$S_{\a\b}$) is
\begin{equation*}
N-F_N\leq C_1sN^{3/5}(\ln N)^{C_2}.
\end{equation*}
Hence, we have
\begin{equation*}
N-C_1sN^{3/5}(\ln N)^{C_2}\leq F_N\leq N
\end{equation*}
which implies that $F_N$ is asymptotic to $N$ as $N\to +\infty$. In particular, this implies that there are infinitely many positive integers satisfying (NSC-$S_{\a\b}$), which in turn implies that one can choose $n_{K+1}$ arbitrarily large and satisfying (NSC-$S_{\a\b}$). This concludes the proof of Proposition \ref{prop.nonres}.
\end{proof}

\section{The model equation}\label{Sect3}

 Recall that 
\begin{equation*} 
\mathcal{A}_{k}=\big\{n_{k}-2,n_{k}-1,n_{k}+1,n_{k}+2\big\}, \quad \A=\bigcup_{k=1}^{K}\A_{k},
\end{equation*}
and  the notation \eqref{1.3}.  We assume that $(n_{k})_{k\geq 1}$ is a sequence which satisfies Lemma \ref{lem.couple} and Proposition \ref{prop.nonres}.

Next, we set $\eps=\nu^{1/4}$ and make the change of unknown $v=\eps u$. Therefore $v$ satisfies   
\begin{equation}\label{cauchy*} 
\left\{
\begin{aligned}
&i\partial_t v+\partial_{x}^{2}v  =|v|^{4}v,\quad
(t,x)\in\R\times {\T}^{},\\
&v(0,x)= v_{0}(x)=\eps u_{0}(x).
\end{aligned}
\right.
\end{equation} 
We  expand $v$ and $\bar v$ in Fourier modes
$$v(x)=\sum_{j\in \Z}\xi_j e^{ijx},\quad \bar v(x)=\sum_{j\in\Z} \eta_j e^{-ijx},$$
and denote by $\dis H=\int_{\mathbb{S}^{1}}|\partial_{x}v|^{2}+\frac13 \int_{\mathbb{S}^{1}}|v|^{6}$ the Hamiltonian of \eqref{cauchy*}.

We now refer to \cite[Section 3]{GT2}, and in particular to Proposition 3.1 and Proposition 3.3. It is proven that, thanks to a Birkhoff normal form procedure, that there exists a symplectic change of coordinates $\tau$ close to the identity so that
\begin{equation}\label{ham*}
\ov{H}:=H\circ \tau = N +Z_{6}+R_{10},
\end{equation}
where 
\begin{itemize}
\item $N$ only depends on the actions $(I_{k})_{k\in \Z}$ ;
\item  $\dis Z_6$  is the homogeneous polynomial of degree 6 ;     
\begin{equation*}
Z_{6}=\sum_{\mathcal{R}}\xi_{j_{1}}\xi_{j_{2}}\xi_{j_{3}}\eta_{\ell_{1}}\eta_{\ell_{2}}\eta_{\ell_{3}}.
\end{equation*}
\item  $R_{10}$ is a remainder of order 10.
\end{itemize}
~

As in \cite[Section 4]{GT2}, we  introduce the model system by setting $\xi^{0}_{j}=\eta^{0}_{j}=0$ in \eqref{ham*} when $j\neq \mathcal{A}$. This induces here the Hamiltonian  
\begin{equation*} 
\wh{H}=6J^{3}+\sum_{k=1}^{K}\wh{H}_{k}, 
\end{equation*}
where 
\begin{equation*}
\wh{H}_{k}=\sum_{j\in \mathcal{A}_{k}}j^{2}I_{j}-9J\sum_{j \in \mathcal{A}_{k}}I^{2}_{j}+4\sum_{j\in \mathcal{A}_{k}}I^{3}_{j}+18I_{a^{(k)}_{2}}^{1/2}I_{b^{(k)}_{2}}^{1/2}I_{a^{(k)}_{1}}I_{b^{(k)}_{1}}\cos (2\phi^{(k)}_{0}), 
\end{equation*}
with $\phi^{(k)}_{0} =\theta_{a^{(k)}_{1}}-\theta_{b^{(k)}_{1}}+\frac12\theta_{a^{(k)}_{2}}-\frac12\theta_{b^{(k)}_{2}}$. \\
 
Since   $\wh{H}$ is almost decoupled, we obtain a completely integrable system.

\begin{lemm}
The system given by $\wh{H}$ is completely integrable.
\end{lemm}

\begin{proof}
Since  $J$ is a constant of motion, this is a direct consequence of \cite[Lemma 4.1]{GT2}. Indeed, for all $1\leq k\leq K$, $\wh{H}_{k}$ and 
\begin{equation*}
K^{(k)}_{1}=I_{a^{(k)}_{1}}+I_{b^{(k)}_{1}},\quad K^{(k)}_{2}=I_{a^{(k)}_{2}}+I_{b^{(k)}_{2}} \;\;\text{and}\;\;K^{(k)}_{1/2}=I_{b^{(k)}_{2}}+\frac12 I_{a^{(k)}_{1}},
\end{equation*} 
are constants of motion in involution.
\end{proof}

 As in \cite{GT2}, we use the coordinates $(\phi,K)$ to describe some particular trajectories of $\wh{H}$. To begin with, $\wh{H}_{k}$ can be rewritten
\begin{multline*}
\wh{H}_{k}=\wh{H}_{k}(\phi^{(k)}_{0},K^{(k)}_{0},K^{(k)}_{1},K^{(k)}_{2},K^{(k)}_{1/2})\\
\begin{aligned}
&=F_{k}(K^{(k)}_{1},K^{(k)}_{2},K^{(k)}_{1/2})+6\big[(3J-2K_{1})I_{a_{1}}I_{b_{1}}+(3J-2K_{2})I_{a_{2}}I_{b_{2}}+3I_{a_{2}}^{\frac12}I_{b_{2}}^{\frac12}I_{a_{1}}I_{b_{1}}\cos(2\phi_{0})\big],
\end{aligned}
\end{multline*}
for some  polynomial $F_{k}$\footnote{The expression of $F_{k}$ which can be found in \cite{GT2} is incomplete and it lacks the dependence of $F_k$ on $K^{(k)}_{1/2}$. However, this does not affect the results in \cite{GT2}.}. Next, we define
\begin{equation}\label{def.eps}
\eps_{k}=\eps \e^{-n_{k}},
\end{equation}
we fix
\begin{equation*}
K^{(k)}_{1}=\eps^{2}_{k},\quad K^{(k)}_{2}=\frac12\eps^{2}_{k}\; \;\;\text{and}\;\;\;K^{(k)}_{1/2}=\frac12\eps^{2}_{k},
\end{equation*} 
and denote by  $K^{(k)}_{0}=I_{a^{(k)}_{1}}$. With the previous choice, 
\begin{equation}\label{def.J}
J=\sum_{k\geq 1}(K^{(k)}_{1}+K^{(k)}_{2})=\frac32 \eps^{2}\sum_{k\geq1} \e^{-2n_{k}}=C\eps^{2}.
\end{equation}
The following rescaling proves to be useful
\begin{equation}\label{renormal}
\phi^{(k)}_{0}(t)=\phi^{(k)}(\eps^{4}_{k}t),\quad K^{(k)}_{0}(t)=\eps^{2}_{k}K^{(k)}(\eps^{4}_{k}t).
\end{equation} 
Define
\begin{equation*}\label{hat*}
H^{(k)}_{\star}=\frac94K^{(k)}(1-K^{(k)})\Big[(10J\eps^{-2}_{k}-6)+4(K^{(k)})^{\frac12}(1-K^{(k)})^{\frac12}\cos(2\phi^{(k)})\Big],
\end{equation*}
and 
\begin{equation}\label{def.Hstar}
H_{\star}=\sum_{k=1}^{K}H^{(k)}_{\star},
\end{equation}
then for all $1\leq k\leq K$, the evolution of $(\phi^{(k)},K^{(k)})$ is given by
\begin{equation} \label{syst*}
\left\{\begin{array}{rcl}
\dot \phi^{(k)}&=\dis -\frac{\partial {H_{\star}}}{\partial K^{(k)}}&=-\frac{27}4(1-2K^{(k)})\Big[(\frac{10}3J\eps^{-2}_{k}-2)+2(K^{(k)})^{\frac12}(1-K^{(k)})^{\frac12}\cos(2\phi^{(k)})\Big] \\[10pt]
\dot K^{(k)}   &=\dis \frac{\partial {H_{\star}}}{\partial \phi^{(k)}}&=-18(K^{(k)})^{\frac32}(1-K^{(k)})^{\frac32}\sin(2\phi^{(k)}).
\end{array}\right.
\end{equation}

Then by \cite[Proposition 4.2]{GT2} we have
\begin{prop}\label{prop.mar}~
For all $k\geq 1$, there exists a $\gamma_{k}\in (0,1/2)$ so that if $\gamma_{k}<K^{(k)}(0)<1-\gamma_{k}$ and $\phi(0)=0$, then there is $T_{k}>0$ so that $(\phi^{(k)},K^{(k)})$ is a $2T_{k}-$periodic solution of \eqref{syst*} and 
\begin{equation*} 
K^{(k)}(0)+K^{(k)}(T_{k})=1.
\end{equation*}  
Moreover, $ T_{k} : (\gamma_{k},1/2)\cup(1/2,1-\gamma_{k}) \longrightarrow \R$ is a continuous function of $K^{(k)}(0)$, 
\begin{equation*}
T_{k}\longrightarrow\frac{2\pi}{9}(10 J\eps^{-2}_{k}-3)^{-1/2}\quad \text{as}\quad K^{(k)}(0) \longrightarrow 1/2,
\end{equation*}
and 
\begin{equation*}
T_{k}\longrightarrow +\infty\quad \text{as}\quad K^{(k)}(0) \longrightarrow \gamma_{k}.
\end{equation*}
\end{prop}

\begin{proof}
Denote by $C_{k}=10 J\eps^{-2}_{k}-6$. With \eqref{def.J}, we have $C_{k}\geq 9$ and $\dis C_{k}\sim c\e^{2n_{k}}$. Write here $K^{(k)}=K$ and $\phi^{(k)}=\phi$. Then, as for the proof of \cite[Proposition 4.2]{GT2},
\begin{equation*} 
K(1-K)\big(C_{k}+4K^{\frac12}(1-K)^{\frac12}\cos(2\phi)\big)=\frac{1}{4}(C_{k}-2)
\end{equation*}
defines two heterocline orbits which link the saddle points $(\phi,K)=(-\pi/2,1/2)$ and $(\phi,K)=(-\pi/2,1/2)$.
When $\phi=0$
\begin{equation*}
K^{\frac12}(1-K)^{\frac12}=D_{k}:=\frac{C_{k}-2}{\sqrt{(C_{k}-2)^{2}+8(C_{k}-2)}+C_{k}-2}.
\end{equation*}
Observe that, since $C_k\geq9>2$, $D_{k} \in (0,1/2)$, and this yields the existence of $\gamma_{k}$.

Moreover, $T_{k}$ is a continuous function of $K^{(k)}(0)$ on $(\gamma_{k},1/2)\cup(1/2,1-\gamma_{k})$ and $T_{k}\longrightarrow +\infty$ when $K^{(k)}(0) \longrightarrow \gamma_{k}$, due to the continuous dependence on initial conditions.

Finally, since the point $(\phi,K)=(0,1/2)$ is a nondegenerate centre, we have that $T_k$ approaches the value of the half-period of the linearized system as $K^{(k)}(0) \longrightarrow 1/2$ (it follows from the Theorem in Chapter 5.C, p.100, \cite{Arn89}). Since, in particular, we have
\begin{equation*} 
\frac{\partial^2 H_{\star}}{\partial K^2}(0,1/2)=-\frac{9}{2}(C_k+3), \qquad \frac{\partial^2 H_{\star}}{\partial\phi^2}(0,1/2)=-\frac{9}{2}, \qquad \frac{\partial^2 H_{\star}}{\partial K\partial\phi}(0,1/2)=0,
\end{equation*}
we deduce
\begin{equation}\label{limit.period}
\lim_{K^{(k)}(0) \longrightarrow 1/2}T_k=\frac{2\pi}{9\sqrt{C_k+3}}.
\end{equation}

\end{proof}

\section{The perturbation analysis and proof of the main results}\label{Sect4}
We are now ready to complete the proof of Theorem \ref{thm1}. This is a direct application of the results of \cite[Section 5]{GT2}.\\
We now assume that   $n_{K}\leq -(\ln \eps)/8=-(\ln \nu)/32$, therefore by \eqref{def.eps} we have for all $k\geq 1$
\begin{equation*}
\eps^{9/8}\leq \eps_{k}\leq \eps.
\end{equation*}
Then we can state the following result, which is  analogous to \cite[Lemma 5.1]{GT2}.

 \begin{lemm}\label{lem.dyn}
 Assume that there exists $C>0$ so that 
 \begin{equation*} 
| \xi_{j}(0)|,\;|\eta_{j}(0)|\leq C \eps \e^{-n_{k}},\;\forall \,j\in \mathcal{A}_{k}
 \end{equation*}
 and 
  \begin{equation*} 
|\xi_{p}(0)|,\;|\eta_{p}(0)|\leq C\eps^{3} \e^{-|p| },\;\forall \,p\not\in \mathcal{A}.
  \end{equation*}
 
Then for all $0\leq t \leq C\eps^{-6}$, 
 \begin{equation*} 
 I_{p}(t)=\mathcal{O}(\eps^{6})\;\;\text{when}\;\; p\not\in \mathcal{A},
 \end{equation*}
and
  \begin{eqnarray*}
 K^{(k)}_{1}(t)&=&K^{(k)}_{1}(0)+\mathcal{O}(\eps^{10 })t \\
 K^{(k)}_{2}(t)&=&K^{(k)}_{2}(0)+\mathcal{O}(\eps^{10 })t\\
 K^{(k)}_{1/2}(t)&=&K^{(k)}_{1/2}(0)+\mathcal{O}(\eps^{10 })t.
 \end{eqnarray*}
 \end{lemm}

We consider the  initial conditions
 \begin{equation} \label{inic}
\begin{array}{l}
\dis K^{(k)}_{1}(0)=\eps_{k}^{2},\;\;K^{(k)}_{2}(0)=\frac12\eps_{k}^{2},\;\;K^{(k)}_{1/2}(0)=\frac12\eps_{k}^{2}, \\[8pt]
\text{and} \;\;|\xi_{j}(0)|, |\eta_{j}(0)| \leq C\eps^{3}\e^{-|j|}\;\; \text{for} \;\;j\notin \mathcal{A},
\end{array} 
\end{equation}
and for all $1\leq k\leq K$ we set $\tau_{k}=\eps^{4}_{k}t$. Then thanks to Lemma \ref{lem.dyn} we have 
\begin{prop}\label{prop.4.2} Consider the solution of the Hamiltonian system given by $\ov{H}$  with the initial conditions \eqref{inic}. Then for all $1\leq k\leq K$, $(\phi^{(k)},K^{(k)})$ defined by \eqref{renormal} satisfies for $0\leq \tau_{k} \leq \eps^{-6}\eps^{4}_{k}$
\begin{equation*} 
\left\{\begin{array}{rr}
\dot \phi^{(k)}(\tau_{k})=&-\frac{\partial H_{\star}}{\partial K^{(k)}}+\mathcal{O}(\eps^{2})\\[5pt]
\dot K^{(k)}(\tau_{k})=&\frac{\partial H_{\star}}{\partial \phi^{(k)}} +\mathcal{O}(\eps^{2})  ,
\end{array}\right. 
\end{equation*}
where $H_{\star}$ is the Hamiltonian \eqref{def.Hstar}.
\end{prop}

\subsection{Proof of Theorem \ref{thm1}}

We choose the   initial conditions for $(\phi,K)$. We take $\phi^{(k)}(0)=0$ and $\gamma_{k}<K^{(k)}(0)<1-\gamma_{k}$. We also  consider the solution $(\phi^{(k)}_{\star},K^{(k)}_{\star})$ to \eqref{syst*} with initial condition  $(\phi^{(k)}_{\star},K^{(k)}_{\star})(0)= (\phi^{(k)},K^{(k)})(0)$. Then by Proposition \ref{prop.4.2} and \cite[Lemma 5.3]{GT2} \footnote{As here, in  \cite[Lemma 5.3]{GT2} there is actually an additional error term $\mathcal{O}(\eps^2)\tau^{2}$. This  restrains the main  result in  \cite{GT2} to times $t\leq \eps^{-5}=\nu^{-5/4}$.} for all $0\leq \tau_{k} \leq \eps^{-5}\eps^{4}_{k}$ we have
\begin{equation*} 
 (\phi^{(k)},K^{(k)})(\tau_{k})=(\phi^{(k)}_{\star},K^{(k)}_{\star})(\tau_{k})+ \mathcal{O}(\eps^2)\tau_{k}+\mathcal{O}(\eps^2)\tau_{k}^2,
\end{equation*}
 which in turn implies that for all $0\leq t\leq \eps^{-5}$  
\begin{eqnarray*}
K^{(k)}_{0}(t)&=&\eps_{k}^{2}K^{(k)}_{\star}(\eps_{k}^{4}t)+ \mathcal{O}(\eps^{2}\eps^{6}_{k})t+\mathcal{O}(\eps^{2}\eps^{10}_{k})t^2  \\
\phi^{(k)}_{0}(t)&=&\phi^{(k)}_{\star}(\eps_{k}^{4}t)+ \mathcal{O}(\eps^{2}\eps^{4}_{k})t+\mathcal{O}(\eps^{2}\eps^{8}_{k})t^2.
\end{eqnarray*}
For $t\leq \eps^{-9/2}=\nu^{-9/8}$ we get \eqref{u}.

The period of $K^{(k)}_{0}$ is $2T_{k}\eps^{-4}_{k}=2T_{k}\eps^{-4}\e^{4n_{k}}$, thus one has to ask that 
\begin{equation}\label{log.cond}
2T_{k}\e^{4n_{k}}\leq \eps^{-1/2}=\nu^{-1/8}.
\end{equation}
In view of Proposition \ref{prop.mar}, we have that, depending on the choice of initial data, $T_k$ spans at least the open interval $(\frac{2\pi}{9}(10 J\eps^{-2}_{k}-3)^{-1/2},+\infty)$, which contains the interval $({\pi}/(9\sqrt{3}),+\infty)$ for all $k$, so we can choose the initial data in such a way that $T_k=1/2$ for all $k$. Therefore, since $n_k\leq n_K$, \eqref{log.cond} is satisfied provided that $n_K\leq-1/32\ln\nu$.

\subsection{Proof of Corollary \ref{Coro}}

For each $k$, the period of $v_k$ equals the period of $K^{(k)}_{0}$, whose value is $2\nu^{-1} T_k\e^{4n_k}$. In the proof of Theorem \ref{thm1} we have observed that, given any $\bar T\in({\pi}/(9\sqrt{3}),+\infty)$, one can choose the initial data in such a way that $T_k=\bar T$. This, with $n_k\leq n_K$, implies that, for any given $\bar\Lambda\in({\pi\e^{4n_K}}/(9\sqrt{3}),+\infty)$, one can choose the initial data so that the period of $v_k$ is equal to $2\bar\Lambda/\nu$. Therefore, given any $K\geq 1$ and any $\Lambda_{1},\Lambda_{2},\dots,\Lambda_{K}>0$, there exists $N \in\N$ so that, with a proper choice of the initial data, $v_{k}$ has period $2N\Lambda_{k}/\nu$, if $\nu$ is small enough. It suffices to choose the smallest $N$ such that $N\Lambda_k>\pi\e^{4n_K}/(9\sqrt{3})$ for all $k$, which is admissible if for the chosen $N$ one has $N\Lambda_k<\nu^{-1/8}$ for all $k$, which is verified provided that $\nu$ is small enough.


\begin{thebibliography}{99}

\bibitem{Arn89}
V.~I. {Arnold}.
\newblock {\em {Mathematical Methods of Classical Mechanics}}.
\newblock Springer, Berlin, 1989. Second edition. 
  
\bibitem{BZ02}
E.~Bombieri and U.~Zannier.
\newblock  A note on squares in arithmetic progressions, II.
\newblock{\em Rend. Mat. Acc. Lincei} (2002), s. 9, v. 13, 69--75. 
  
 
 


\bibitem{Gre07}
B. Gr{\'e}bert.
 \newblock Birkhoff normal form and {H}amiltonian PDEs, 
\newblock{\em  Partial differential equations and applications, S\'emin. Congr.,
  vol.~15, Soc. Math. France, Paris} 2007, pp.~1--46. 

      
     \bibitem{GPT}
 B. Gr\'ebert, E. Paturel  and L. Thomann.
 \newblock Beating effects in cubic Schr\"odinger systems and growth of Sobolev norms.
    \newblock{  arXiv:1208.5680}. 

   \bibitem{GT2}
 B. Gr\'ebert  and L. Thomann.
 \newblock Resonant dynamics for the quintic non linear Schr\"odinger equation. 
    \newblock{\em Ann. I. H. Poincar\'e - AN},   29 (2012), no. 3, 455--477.
    
    
 \bibitem{GV}
 B. Gr\'ebert  and C. Villegas-Blas.
 \newblock On the energy exchange between resonant modes in nonlinear Schr\"odinger equations.
  \newblock{\em Ann. I. H. Poincar\'e - AN}, 28 (2011), no. 1, 127--134.
  

\end{thebibliography}
\end{document}